\documentclass[reqno,a4paper]{amsart}
\usepackage{amscd,amssymb,amsmath,mathrsfs,latexsym,url}
\usepackage[all]{xy}
\usepackage{color,graphicx}
\usepackage{hyperref}
\hypersetup{breaklinks=true}

\def\ov#1{{\overline{#1}}}

\def\wt#1{{\widetilde{#1}}}
\def\?{\ ???\ \immediate\write16{}%
\immediate\write16{Warning: There was still a question mark . . . }%
\immediate\write16{}}

\newcommand{\conv}{\operatorname{conv}}

\newcommand{\vol}{\operatorname{vol}}

\newcommand{\dd}{\hspace{1pt}\operatorname{d}\hspace{-1pt}}

\renewcommand{\div}{\operatorname{div}}

\renewcommand{\t}{\operatorname{t}}
\newcommand{\KK}{\operatorname{K}}

\newcommand{\Spec}{\operatorname{Spec}}

\newcommand{\chern}{\operatorname{c}}

\renewcommand{\i}{\operatorname{i}}

\newcommand{\e}{{\rm e}}
\newcommand{\h}{\operatorname{h}}

\newcommand{\Hom}{\operatorname{Hom}}

\newcommand{\m}{\operatorname{m}}
\newcommand{\ord}{{\operatorname{ord}}}

\newcommand{\fin}{{{\rm fin}}}
\newcommand{\hor}{{{\rm hor}}}
\newcommand{\an}{{\rm an}}
\newcommand{\gen}{{{\rm gen}}}

\renewcommand{\vert}{{\rm vert}}

\newcommand{\can}{{\rm can}}

\newcommand{\field}{\operatorname{K}}

\newcommand{\KB}{\K}

\def \C{\mathbb{C}}

\def \G{\mathbb{G}}
\def \K{\mathbb{K}}

\def \P{\mathbb{P}}
\def \Q{\mathbb{Q}}
\def \R{\mathbb{R}}
\def \SS{\mathbb{S}}

\def \T{\mathbb{T}}
\def \Z{\mathbb{Z}}

\def\cB {{\mathcal B}}
\def\cC {{\mathcal C}}

\def\cH {{\mathcal H}}

\def\cL {{\mathcal L}}

\def\cO {{\mathcal O}}

\def\cV {{\mathcal V}}
\def\cW {{\mathcal W}}
\def\cX {{\mathcal X}}
\def\cY {{\mathcal Y}}

\def\fM{{\mathfrak M}}

\newcommand{\bfm}{{\boldsymbol{m}}}

\newcommand{\bft}{{\boldsymbol{t}}}

\newcommand{\bfgamma}{{\boldsymbol{\gamma}}}

\newcommand{\bfzero}{\boldsymbol{0}}

\newcommand{\fp}{\mathfrak{p}}

\numberwithin{equation}{section}
\theoremstyle{definition}
\newtheorem{defn}{Definition}
\numberwithin{defn}{section}

\newtheorem{rem}[defn]{Remark}
\newtheorem{exmpl}[defn]{Example}

\theoremstyle{plain}

\newtheorem{prop}[defn]{Proposition}
\newtheorem{thm}[defn]{Theorem}
\newtheorem{cor}[defn]{Corollary}
\newtheorem{prop-def}[defn]{Proposition-Definition}

\begin{document}

\title[Height of varieties over finitely generated fields]{Height of varieties over finitely generated fields}

\author[Burgos Gil]{Jos\'e Ignacio Burgos Gil}
\address{Instituto de Ciencias Matem\'aticas (CSIC-UAM-UCM-UCM3).
  Calle Nicol\'as Ca\-bre\-ra~15, Campus UAB, Cantoblanco, 28049 Madrid,
  Spain} 
\email{burgos@icmat.es}
\urladdr{\url{http://www.icmat.es/miembros/burgos}}
\author[Philippon]{Patrice Philippon}
\address{Institut de Math{\'e}matiques de
Jussieu -- U.M.R. 7586 du CNRS, \'Equipe de Th\'eorie des Nombres.
BP 247, 4
place Jussieu, 75005 Paris, France}
\email{patrice.philippon@imj-prg.fr}
\urladdr{\url{http://www.math.jussieu.fr/~pph}}
\author[Sombra]{Mart{\'\i}n~Sombra}
\address{ICREA \& Universitat de Barcelona, Departament d'{\`A}lgebra i Geometria.
Gran Via 585, 08007 Bar\-ce\-lo\-na, Spain}
\email{sombra@ub.edu}
\urladdr{\url{http://atlas.mat.ub.es/personals/sombra}}

\thanks{ Burgos Gil was partially supported by the MICINN research
  project MTM2010-17389. Philippon was partially supported by the
  CNRS international project for scientific cooperation (PICS) ``G\'
  eom\' etrie diophantienne et calcul formel'' and the ANR research
  project ``Hauteurs, modularit\'e, transcendance''.  Sombra was
  partially supported by the MINECO research project
  MTM2012-38122-C03-02.}

\date{\today} 

\begin{abstract}
  We show that the height of a variety over a finitely generated field
  of characteristic zero can be written as an integral of local
  heights over the set of places of the field. This allows us to apply
  our previous work on toric varieties and extend our combinatorial
  formulae for the height to compute some arithmetic intersection
  numbers of non toric arithmetic varieties over the rational
  numbers. 
\end{abstract}

\maketitle

\overfullrule=0.9mm

\section*{Introduction} \label{sec:introduction}

In \cite{Moriwaki:ahffgf,Moriwaki:cahsavfgf}, Moriwaki introduced a
notion of height for cycles over a finitely generated extension of
$\Q$. Using this definition, it was possible to  extend several central results about
cycles over a number field to cycles over a finitely
generated extension of $\Q$. These results  include  Northcott's theorem on the
finiteness of cycles with bounded degree and height, the
Manin-Mumford and the Bogomolov conjectures, Zhang's theorem on
successive algebraic minima, and the equidistribution properties of Galois
orbits of points of small height \cite{Moriwaki:ahffgf,
  Moriwaki:cahsavfgf, Moriwaki:gcblfgf, YZ:ahitaln2:sppas}.

This notion of height is defined as follows.  Let $\cB$ be an
arithmetic variety, that is, a normal flat projective scheme over
$\Spec(\Z)$, of relative dimension $b$. Set $\K=\KK(\cB)$ for its
function field, which is a finitely generated extension of $\Q$ of
transcendence degree $b$. Let $\ov \cH_{i}$, $i=1,\dots, b$, be a family
of nef Hermitian line bundles on $\cB$.

Let $\pi\colon \cX \to \cB$ be a dominant morphism of arithmetic
varieties and denote by $X$ the fibre of $\pi$ over the generic point
of $\cB$, which is a variety over $\K$.  Let $Y$ be a prime cycle of
$X$ of dimension $d$. Let $\cY$ be the closure of $Y$ in $\cX$ and
$\ov \cL_{j}$, $j=0,\dots,d$, a family of semipositive Hermitian line
bundles on $\cX$. Moriwaki defines the height of $Y$ relative to this
data as the arithmetic intersection number in the sense of
Gillet-Soul\'e given by
  \begin{equation}\label{eq:25}
    \h_{\pi ^{\ast}\ov \cH_{_{1}},\dots,\pi ^{\ast}\ov \cH_{_{b}},\ov \cL_{0},\dots, \ov \cL_{d}}(\cY),
  \end{equation}
see \cite{Moriwaki:ahffgf,Moriwaki:cahsavfgf} for details.

It is well known that this arithmetic intersection number can be
written as a sum over the places of $\Q$ of local heights of the fibre
of $\cY$ over the generic point of $\Spec(\Z) $, see for instance
\cite[\S~1.5]{BurgosPhilipponSombra:agtvmmh}.  However, for points in
a projective space and the canonical metric, Moriwaki showed that this
arithmetic intersection number is also equal to an integral of local
heights over a measured set of places of~$\K$ \cite[Proposition
3.2.2]{Moriwaki:ahffgf}.

In this paper, we extend Moriwaki's result to a cycle $Y$ of arbitrary
dimension and general semipositive metrics, to show that the height of
$Y$ is equal to an integral of local heights over this set of places
of $\K$ (Theorem \ref{thm:1}).

This allows us to apply our previous work on toric varieties in
\cite{BurgosPhilipponSombra:agtvmmh} and extend our combinatorial
formulae for the height to some arithmetic intersection numbers of non
toric arithmetic varieties. More explicitly, let $\cX \to\cB$ be a
dominant morphism of arithmetic varieties as above, and such that its
generic fibre $X$ is a toric variety over $\K$ of dimension $n$. For
simplicity, suppose that $\ov \cL_{0}=\dots=\ov \cL_{n}=\ov \cL$. The
semipositive Hermitian line bundle $\ov \cL$ defines a polytope
$\Delta$ in a linear space of dimension $n$ and, for each place $w$ of
$\K$, a concave function $\vartheta_{w}\colon \Delta\to \R$ called the
$w$-adic ``roof function'', see
\S~\ref{sec:toric-varieties} for details and pointers to the
literature.  By \cite[Theorem 5.1.6]{BurgosPhilipponSombra:agtvmmh},
the $w$-adic local height of $X$ is given by $(n+1)!$ times the
integral over $\Delta$ of this concave function.  Combining this
 with Theorem \ref{thm:1}, we derive a formula for the
corresponding arithmetic intersection number as an integral of the
function $(w,x)\mapsto \vartheta_{w}(x)$ over the product of the
polytope and the set of places of $\K$ (Corollary \ref{cor:1}).
Furthermore, we can define a global roof function $\vartheta\colon
\Delta\to \R$ by integrating the local ones over the
 set of places of $\K$. Then, in this case, the arithmetic intersection number
 \eqref{eq:25} is also equal to $(n+1)!$ times the integral of
 $\vartheta$ over the polytope (Corollary \ref{cor:2}).

 As an application, we give in \S~\ref{sec:height-transl-subt} an
 explicit formula for the case of translates of subtori of a
 projective space and canonical metrics (Corollary \ref{cor:3}).  The
 obtained integrals reduce, in some instances, to logarithmic Mahler
 measures of multivariate polynomials.

\medskip \noindent {\bf Acknowledgements.} 
Part of this work was done while the authors met at the Universitat de
Barcelona, the Instituto de Ciencias Matem\'aticas (Madrid) and the
Institut de Math\'ematiques de Jussieu (Paris). We are thankful for
their hospitality. We also thank
Julius Hertel and the referee for their useful comments.

\section{Fields with product formula from arithmetic
  varieties} \label{sec:general-varieties}

In \cite[Example 11.22]{Gubler:lchs}, Gubler observed that, for an
arithmetic variety equipped with a family of nef Hermitian line
bundles, one can endow its function field with a measured set of
places satisfying the product formula.  In this section, we explain
the details of this construction and, as an example, we explicit it
for the projective space and the universal line bundle equipped with the canonical metric. We
refer to \cite[Chapter 1]{BurgosPhilipponSombra:agtvmmh} and
\cite[\S~3]{BurgosMoriwakiPhilipponSombra:aptv} for the background for
this section on metrized line bundles and their associated measures
and heights.

Let $\cB$ be an \emph{arithmetic variety}, which means that $\cB$ is a
normal flat projective scheme over $\Spec(\Z)$. We denote by $b$ the
relative dimension of $\cB$ and by $\KB=\field(\cB)$ its function
field, which is a finitely generated extension of $\Q$ of
transcendence degree $b$.  For $i=1, \dots, b$, let $\ov
\cH_{i}=(\cH_{i},\|\cdot\|_{i})$ be a \emph{Hermitian line bundle} on
$\cB$, that is, a line bundle $\cH_{i}$ on $\cB$ equipped with a
continuous metric on the complexification $\cH_{i,\C}$ over $\cB(\C)$,
invariant under complex conjugation.  We will furthermore assume that
each $\ov \cH_{i}$ is \emph{nef} in the sense of
\cite[\S~2]{Moriwaki:ahffgf} or \cite[Definition
3.18(3)]{BurgosMoriwakiPhilipponSombra:aptv}.  This amounts to the
conditions:
\begin{enumerate}
\item \label{item:6}  the metric $\|\cdot\|_{i}$ is \emph{semipositive},
namely it is the uniform limit of a sequence of smooth semipositive
metrics as in \cite[Definition 4.5.5]{Maillot:GAdvt} or
\cite[Definition 1.4.1]{BurgosPhilipponSombra:agtvmmh}; 
\item \label{item:7}  the
height of every integral one-dimensional subscheme of $\cB$ with
respect to $\ov\cH_{i}$ is nonnegative.
\end{enumerate}

Let $\cB^{(1)}$ denote the set of hypersurfaces of $\cB$, that is, the
integral subschemes of $\cB$ of codimension 1. Let $\cV\in \cB^{(1)}$.
By \cite[Theorem 1.4(a)]{Zhang:_small} or
\cite[Proposition~2.3]{Moriwaki:ahffgf}, the hypothesis that the
$\cH_{i}$'s are nef implies that the height of $\cV$ with respect to
these Hermitian line bundles, denoted by $\h_{\ov
  \cH_{1},\dots,\ov\cH_{b}}(\cV)$, is nonnegative.  Hence, we
associate to $\cV$ the non-Archimedean absolute value on $\KB$ given,
for $\gamma\in \K$, by
\begin{equation*}
 |\gamma |_{\cV}=\e^{-\h_{\ov \cH_{1},\dots,\ov\cH_{b}}(\cV)\ord_{\cV}(\gamma )},
\end{equation*}
where $\ord_{\cV}$ denotes the discrete valuation associated to the
local ring $\cO_{\cB,\cV}$. We denote by $\mu_{\fin} $ the counting
measure of $\cB^{(1)}$.

We define the \emph{set of generic points} of $\cB(\C)$
as
\begin{equation*}
  \cB(\C)^{\gen}=\cB(\C)\setminus \bigcup_{\cV\in
    \cB^{(1)}}\cV(\C). 
\end{equation*}
By definition, a point $p\in \cB(\C)$ belongs to $\cB(\C)^{\gen}$ if
and only if, for all $\gamma \in \KB^{\times}$, this point does not
lie in the analytification of the support of $\div(\gamma )$. Hence
$|\gamma (p)|$ is a well-defined positive real number, and we
associate to $p$ the Archimedean absolute value given, for $\gamma\in
\K^{\times}$, by
\begin{equation}
  \label{eq:3}
 |\gamma |_{p}=|\gamma (p)|.
\end{equation}
On $\cB(\C)$, we consider the  measure 
\begin{displaymath}
\mu _{\infty}=  \chern_{1}(\ov \cH_{1})\land
\dots\land \chern_{1}(\ov \cH_{b})
\end{displaymath}
associated to the family of semipositive Hermitian line bundles $\ov
\cH_{i}$, $i=1,\dots, b$, as in \cite[Definition
1.4.2]{BurgosPhilipponSombra:agtvmmh}.  By \cite[Corollaire
4.2]{ChambertLoirThuillier:MMel}, the measure of each hypersurface of
$\cB(\C)$ with respect to $\mu _{\infty}$ is zero.  Since the
complement of $\cB(\C)^{\gen}$ is a countable union of hypersurfaces,
it has measure zero. We will denote also by $\mu _{\infty}$ the
induced measure on $\cB(\C)^{\gen}$.

Put then
\begin{equation}
  \label{eq:6}
  (\fM,\mu )=(\cB^{(1)},\mu_{\fin} )\sqcup (\cB(\C)^{\gen},\mu _{\infty}).
\end{equation}
The set $\fM$ is in bijection with a set of absolute values. Moreover,
all the non-Archimedean absolute values in this set are associated to
a discrete valuation.

\begin{exmpl}
  \label{exm:1}
  Let $\cB=\P^{b}_{\Z}$ with projective coordinates
  $(x_{0}:\dots:x_{b})$ and $\ov \cH_{i}=\ov{\cO(1)}^{\can}$,
  $i=1,\dots, b$, the universal line bundle on $\P_{\Z}^{b}$ equipped
  with the canonical metric as in \cite[Example
  1.4.4]{BurgosPhilipponSombra:agtvmmh}.  We have that $\K=\KK( \cB)\simeq
  \Q(z_{1},\dots, z_{b})$, with $z_{i}=x_{i}/x_{0}$.

  Consider the compact subtorus of
  $\P^{b}_{\Z}(\C)$ given by
  \begin{displaymath}
\SS=\{ (1:z_{1}:\dots:z_{b})\in \P^{b}_{\Z}(\C) \mid |z_{i}|=1 \text{
  for all } i\} \simeq (S^{1})^{b}
  \end{displaymath}
  and the measure $\mu _{\SS}$ of 
  $\P^{b}(\C)$ given by the current
  \begin{displaymath}
    \frac{1}{(2\pi i)^{b}}\frac{\dd z_{1}}{z_{1}}\wedge \dots\wedge
    \frac{\dd z_{b}}{z_{b}}\wedge \delta _{\SS}.
  \end{displaymath}
  Namely, $\mu_{\SS}$ is the Haar probability measure on $\SS$. 

  A hypersurface $\cV$ of $\P_{\Z}^{b}$ corresponds to an irreducible
  homogeneous polynomial $P_{\cV}\in \Z[x_{0},\dots, x_{b}]$.  The
  associated absolute value is given, for $\gamma\in \K^{\times}$, by
  \begin{equation*}
\log |\gamma|_{\cV}= {-\ord_{\cV}(\gamma)\m(P_{\cV})},
  \end{equation*} 
  where $\m(P_{\cV})$ is the logarithmic Mahler measure of
  $P_{\cV}$ given by
  \begin{displaymath}
    \m(P_{\cV})=\int \log|P_{\cV}(1,z_1,\dots,z_b)| \dd \mu_{\SS}.
  \end{displaymath}
  If $P_{\cV}$ is a irreducible polynomial of degree
  zero, then $P_{\cV}=p\in \Z$, a prime number. In this case
  $\m(P_{\cV})=\log(p)$ and $\cV$ is the fibre over the point
  corresponding to $p$.

  The absolute value associated to a point of $\P_{\Z}^{b}(\C)^{\gen}$
  is given by the Archimedean absolute value of the evaluation at this
  point as in \eqref{eq:3}.

  In this example, the measure $\mu_{\fin}$ on $(\P^{b}_{\Z})^{(1)}$ is the counting
  measure and the measure
  $\mu_{\infty}$ is the restriction to $\P_{\Z}^{b}(\C)^{\gen}$ of $\mu_{\SS}$.
\end{exmpl}

A function on a measured space is \emph{integrable} (also called
\emph{summable}) if its integral is a well-defined real number. 

\begin{prop}\label{prop:1} 
  For each $\gamma \in \KB^{\times}$, the function $\fM\to \R$ given
  by $w\mapsto \log |\gamma |_{w}$ is $\mu$-integrable. Furthermore, the ``product formula''
  \begin{equation}\label{eq:18}
    \int_{\fM}\log|\gamma |_{w}\dd \mu
    (w)=0
  \end{equation}
holds.  
\end{prop}
\begin{proof} Given $\gamma\in \K^{\times}$, the set of hypersurfaces
  $\cV$ such that $ |\gamma |_{\cV}\not =1$ is contained in the set of
  components of the support of $\div(\gamma )$, when $\gamma $ is
  viewed as a rational function on $\cB$. Hence this set is finite, and
  so the function on $\cB^{(1)}$ given by $\cV\mapsto
  \log|\gamma|_{\cV}$ is $\mu_{\fin}$-integrable. Moreover, by
  \cite[Th\'eor\`eme 4.1]{ChambertLoirThuillier:MMel}, the function on
  $\cB(\C)^{\gen}$ given by $p\mapsto \log |\gamma (p)|$ is $\mu
  _{\infty}$-integrable.  Summing up, $\log |\gamma |_{w}$ is
  $\mu$-integrable, which proves the first statement.

For the second one, let $\ov \cO$ be the
  trivial metrized line bundle on $\cB$.  Then
  \begin{align*}
    \int_{\fM}\log|\gamma |_{w}\dd \mu
    (w)&=    \sum_{\cV\in \cB^{(1)}}  -\ord_{\cV}(\gamma )\h_{\ov
      \cH_{1},\dots,\ov \cH_{b}}(\cV)+
    \int_{\cB(\C)^{\gen}}\log|\gamma (p)|\dd \mu
    _{\infty}(p)\\ 
    &=-\h_{\ov \cO,\ov     \cH_{1},\dots,\ov \cH_{b}}(\cB)
  \end{align*}
  by the arithmetic B\'ezout formula, see for instance
  \cite[(3.2.2)]{BostGilletSoule:HpvpGf} for the smooth case or
  \cite[Th\'eor\`eme 1.4]{ChambertLoirThuillier:MMel} for an adelic
  version in the general case. From the multilinearity of the height,
  it follows that $\h_{\ov \cO,\ov \cH_{1},\dots,\ov \cH_{b}}(\cB)=0$,
  which concludes the proof.
\end{proof}

\begin{defn}
  \label{def:3}
Given $\bfgamma=(\gamma_{0},\dots, \gamma_{n}) \in \K^{n+1} \setminus
\{\bfzero\}$, the  \emph{size} of $\bfgamma $ with respect to
$(\K,\fM, \mu)$ is defined as
\begin{equation}
  \label{eq:1}
  \t_{\K,\fM, \mu}(\bfgamma)=\int_{\fM}\log\max
  (|\gamma_{0}|_{w},\dots, |\gamma_{n}|_{w}) \dd\mu(w).
\end{equation}
\end{defn}

\begin{exmpl}
  \label{exm:2}
  Let $\K\simeq \Q(z_{1},\dots, z_{b})$ with the measured set of places
  $(\fM,\mu)$ as described in Example \ref{exm:1}.  Let $\gamma\in \K^{\times}$
  given in reduced representation as $\gamma=\alpha/\beta$ with
  coprime $\alpha,\beta\in \Z[z_{1},\dots, z_{b}]$. Using the product
  formula \eqref{eq:18}, the size of $\gamma$ can be given in this
  case by
\begin{displaymath}
  \t_{\K,\fM,\mu}(\gamma)=\int_{\mathfrak{M}}\log \max(|\alpha|_{w},|\beta|_{w})\dd \mu(w)
\end{displaymath}
Since $\alpha$ and $\beta$ are coprime,
the contribution of the integral over the places of
$(\P^{b}_{\Z})^{(1)}$ is zero. Hence,
\begin{equation}\label{eq:33}
  \t_{\K,\fM,\mu}(\gamma)=\frac{1}{(2\pi i)^{b}} \int_{(S^{1})^{b}}\log
  \max(|\alpha(z)|,|\beta(z)|)\frac{\dd z_{1}}{z_{1}}\wedge \dots\wedge \frac{\dd z_{b}}{z_{b}}.
\end{equation}
Using Jensen's formula, this size can be alternatively written as the
logarithmic Mahler measure of the polynomial
\begin{equation}
  \label{eq:26}
P_{\gamma}=\alpha(z_{1},\dots, z_{b}) t_{1}-\beta(z_{1},\dots, z_{b})  \in \Z[t_{1},z_{1},\dots, z_{b}],
\end{equation}
where $t_{1}$ denotes an additional variable. The difference between
this size and the logarithm of the maximum of the absolute values of the
coefficients of $\alpha$ and $ \beta$ can be bounded by the maximum of
their degrees times a constant depending only on~$b$.
\end{exmpl}

\section{Relative arithmetic varieties}
\label{sec:relat-arithm-vari}

In this section we prove our main result (Theorem \ref{thm:1}),
showing that the height of a cycle over the finitely generated
extension $\K$ can be written as an integral of the local heights of
this cycle over the measured set of places $(\fM,\mu)$.

Let $\pi\colon \cX \to \cB$ be a dominant morphism of arithmetic
varieties of relative dimension $n\ge 0$ and $\ov \cL$ a Hermitian
line bundle on $\cX$.  We denote by $X$ the fibre of $\pi$ over the
generic point of $\cB$. This is a variety over $\KB$ of dimension $n$,
and the line bundle $\cL$ induces a line bundle on $X$, denoted by
$L$. There is a  collection of metrics on the
analytifications of $L$ for each absolute value of $\fM$, that we now
describe.
 
For each $\cV\in \cB^{(1)}$, the local ring $\cO_{\cB,\cV}$ is a
discrete valuation ring with field of fractions $\KB$. The scheme
$\cX$ and the line bundle $\cL$ induce a projective model over $\Spec
(\cO_{\cB,\cV})$, denoted $(\cX_{\cV},\cL_{\cV})$, of the pair
$(X,L)$. Following Zhang, the model $(\cX_{\cV},\cL_{\cV})$ induces a
metric on the analytification $L_{\cV}^{\an}$ over $X_{\cV}^{\an}$,
see \cite{Zhang:_small} or \cite[Definition
1.3.5]{BurgosPhilipponSombra:agtvmmh} for details.

The map $\pi$ also induces a map of complex analytic spaces
$\cX(\C)\to \cB(\C)$, that we also denote by $\pi$.  A point $p\in
\cB(\C)^{\gen}$ induces an Archimedean absolute value $|\cdot|_{p}$ on
$\K$ and the analytification of the variety $X$ with respect to
$|\cdot|_{p}$ can be identified with the fibre $\pi ^{-1}(p)\subset
\cX(\C)$, with its structure of real analytic space when the point $p$
is real.  The analytification of the line bundle $L$ on $X$ with
respect to $|\cdot|_{p}$ can also be identified with the restriction
of $\cL_{\C}$ to $\pi ^{-1}(p)$. Then the metric on $L^{\an}_{p}$ is
defined as the restriction of the metric on
$\cL_{\C}$ to this fibre. We then denote
\begin{equation}\label{eq:20}
  \ov L=(L,(\|\cdot\|_{w})_{w\in \fM})
\end{equation}
the obtained \emph{$\fM$-metrized line bundle} on $X$.

Let $Y$ be a $d$-dimensional cycle on $X$ and $\ov{L}_{i}$,
$i=0,\dots, d$, $\fM$-metrized line bundles on $X$ as in
\eqref{eq:20}.  We assume that each $\ov L_{i}$ is constructed from a
\emph{DSP} Hermitian line bundle $\ov\cL_{i}$ on $\cX$. Recall that a
\emph{DSP} (difference of semipositive) Hermitian line bundle on $\cX$ is the quotient of two
semipositive ones as in \cite[Definition
1.4.1]{BurgosPhilipponSombra:agtvmmh}.  

Given a collection of nonzero rational
sections $s_{i}$ of $\cL_{i}$, $i=0,\dots, d$, intersecting properly
on $Y$ and $w\in\fM$,  we denote by
\begin{displaymath}
 \h_{\ov L_{0,w}, \dots,
  \ov L_{d,w}}(Y;s_{0},\dots,s_{d})
\end{displaymath}
the {local height} of $Y$ with respect to the family of $w$-adic metrized line
bundles $\ov L_{i,w}:=(L_{i},\|\cdot\|_{i,w})$, $i=0, \dots, d$. 
It is defined inductively on the dimension of $Y$ by the arithmetic B\'ezout formula
\begin{multline} \label{eq:10}
     \h_{\ov L_{0,w}, \dots,
  \ov L_{d,w}}(Y;s_{0},\dots,s_{d})=\h_{\ov L_{0,w}, \dots,
  \ov L_{d,w}}(Y\cdot \div s_{d};s_{0},\dots,s_{d-1}) 
   \\ -\int_{X_{w}^{\an}}\log\|s_{d}\|_{d,w} \bigwedge_{i=0}^{d-1}
    \chern_{1}(\ov L_{i,w})\wedge \delta _{Y},
\end{multline}
see \cite[Definition 1.4.11]{BurgosPhilipponSombra:agtvmmh}. Recall that
the local height
We will show in Theorem \ref{thm:1} below that the function $\fM\to \R$
given, for $w\in\fM$, by
\begin{equation}\label{eq:21}
  w\longmapsto \h_{(L_{0},\|\cdot\|_{0,w}), \dots,
    (L_{d},\|\cdot\|_{d,w})}(Y;s_{0},\dots,s_{d})
\end{equation}
is $\mu$-integrable. 

\begin{defn}\label{def:4}
  With notation as above, the \emph{global height} of $Y$ with respect
  to the $\fM$-metrized line bundles $\ov L_{i}$, $i=0,\dots, d$, is
  defined as the integral of the function in \eqref{eq:21}, that is
\begin{equation*}
  \h_{\ov L_{0},\dots, \ov L_{d}}(Y)= \int_{\fM}
  \h_{(L_{0},\|\cdot\|_{0,w}), \dots,
    (L_{d},\|\cdot\|_{d,w})}(Y;s_{0},\dots,s_{d}) \dd\mu(w).
\end{equation*}  
\end{defn}

Thanks to the product formula, this notion does not depend on the
choice of the sections $s_{i}$.


\begin{exmpl}
\label{exm:4}
Let $\cB$ be an arithmetic variety as above and $(\K,\fM,\mu)$ the
associated finitely generated field and measured set of places. Let 
\begin{displaymath}
\cX=\P_{\cB}^{n}\simeq \cB \times \P_{\Z}^{n} \quad \text{ and } \quad
\ov \cL=
\varpi^{*}\ov{\cO_{\P^{n}_{\Z}}(1)}^{\can},   
\end{displaymath}
where $\varpi$ denotes the projection $ \cB \times \P_{\Z}^{n} \to
\P_{\Z}^{n}$.  Hence $X=\P^{n}_{\K}$ and $\ov L=
\ov{\cO_{\P^{n}_{\K}}(1)}^{\can}$. For a point
$p=(\gamma_{0}:\dots:\gamma_{n})\in X(\K)= \P^{n}_{\K}(\K)$, we have
\begin{equation}
  \label{eq:23}
\h_{\ov L}(p)=   \t_{\K,\fM, \mu}(\bfgamma), 
\end{equation}
where $ \t_{\K,\fM, \mu}(\bfgamma)$ denotes the size of the vector
$\bfgamma=(\gamma_{0},\dots, \gamma_{n})$ as in \eqref{eq:1}. 
This is the ``naive height'' in \cite[\S~3.2]{Moriwaki:ahffgf}.
\end{exmpl}

The following projection formula for heights of schemes over
$\Spec(\Z)$ generalizes \cite[Proposition 1.3(1)]{Moriwaki:ahffgf}.

\begin{prop} \label{prop:2} Let $\pi \colon \cW\to \cV$ be a morphism
  between two finitely generated projective schemes over $\Spec(\Z)$ of relative
  dimensions $d+b-1$ and 
  $b-1$, respectively, with $b,d\ge0$. Let $\ov \cL_{i}$, $i=1,\dots,
  d$, and $\ov \cH_{j}$, $j=1,\dots, b$, be DSP Hermitian line bundles
  on $\cW$ and $\cV$, respectively.  Then
  \begin{displaymath}
     \h_{\pi ^{\ast}\ov \cH_{_{1}},\dots,\pi ^{\ast}\ov \cH_{_{b}},\ov
       \cL_{1},\dots, \ov \cL_{d}}(\cW)=
   \deg_{\cL_{1},\dots,\cL_{d}}(\cW_{\eta}) \h_{\ov
      \cH_{1},\dots,\ov
      \cH_{b}}(\cV), 
  \end{displaymath}
where $\cW_{\eta}$ denotes the  fibre of $\cW$ over
the generic point $\eta$ of $\cV$. In particular, if $\pi$ is not
dominant, then $\h_{\pi ^{\ast}\ov \cH_{_{1}},\dots,\pi ^{\ast}\ov \cH_{_{b}},\ov
       \cL_{1},\dots, \ov \cL_{d}}(\cW)=0$. 
\end{prop}

\begin{proof}
By linearity, we can reduce  to the case when the
$\cL_{i}$'s are ample and the metrics are semipositive. By continuity,
we can also reduce to the case when the metrics in $\ov{\cL_{i}}$ and
$\ov{\cH}_{j}$ are smooth for all $i,j$. 

We proceed by induction on $d$. The case $d=0$ is given by
\cite[Proposition 1.3(2)]{Moriwaki:ahffgf} in the case when $\pi$ is
dominant and by
\cite[Theorem 1.5.11(2)]{BurgosPhilipponSombra:agtvmmh} in the general
case. Let $d\ge 1$ and choose a nonzero rational section $s_{d}$ of
$\cL_{d}$. Let $\|\cdot\|_{d}$ denote the metric of
$\ov{\cL_{d}}$. By the arithmetic B\'ezout formula,
\begin{multline} \label{eq:17}
     \h_{\pi ^{\ast}\ov \cH_{_{1}},\dots,\pi ^{\ast}\ov \cH_{_{b}},\ov
       \cL_{1},\dots, \ov \cL_{d}}(\cW)= 
     \h_{\pi ^{\ast}\ov \cH_{_{1}},\dots,\pi ^{\ast}\ov \cH_{_{b}},\ov
       \cL_{1},\dots, \ov \cL_{d-1}}(\div(s_{d}))\\ 
-     \int_{\cW(\C)}\log\|s_{d}\|_{d} \bigwedge_{i=1}^{b}
    \chern_{1}(\pi ^{\ast}\ov \cH_{i})\wedge  
\bigwedge_{j=1}^{d-1}\chern_{1}(\ov \cL_{j}).
\end{multline}
Since $\dim(\cV(\C))=b-1$, we have that $\bigwedge_{i=1}^{b}
    \chern_{1}(\ov \cH_{i})=0$. Hence, the measure in the integral in the right-hand
    side of \eqref{eq:17} is zero, and so this integral is zero too. 
Decompose the divisor of $s_{d}$ into its horizontal and vertical
components over $\Spec(\Z)$ as
\begin{displaymath}
  \div(s_{d})= \div(s_{d})_{\hor}+\div(s_{d})_{\vert}.
\end{displaymath}
Write $\div(s_{d})_{\vert}=\sum_{\fp\in\Spec(\Z)}Z_{\fp}$ as a finite sum
of schemes  over the primes.  We have that $
\deg_{\pi ^{\ast}\cH_{_{1}},\dots,\pi ^{\ast}\cH_{_{b}},
  \cL_{1},\dots, \cL_{d-1}}(Z_{\fp})=0$ because
$\dim(\pi(\div(s_{d})_{\vert}))\le b-1$. It follows that
\begin{multline*}
     \h_{\pi ^{\ast}\ov \cH_{_{1}},\dots,\pi ^{\ast}\ov \cH_{_{b}},\ov
       \cL_{1},\dots, \ov \cL_{d-1}}(\div(s_{d})_{\vert})\\= \sum_{\fp\in
       \Spec(\Z)} \log(\fp) \deg_{\pi ^{\ast}\cH_{_{1}},\dots,\pi ^{\ast}\cH_{_{b}},
  \cL_{1},\dots, \cL_{d-1}}(Z_{\fp})=  0.
\end{multline*}
By the inductive hypothesis, 
\begin{equation*}
     \h_{\pi ^{\ast}\ov \cH_{_{1}},\dots,\pi ^{\ast}\ov \cH_{_{b}},\ov
       \cL_{1},\dots, \ov \cL_{d-1}}(\div(s_{d})_{\hor})=
   \deg_{\cL_{1},\dots,\cL_{d-1}}(\div(s_{d})_{\hor,\eta}) \h_{\ov
      \cH_{1},\dots,\ov
      \cH_{b}}(\cV).
\end{equation*}
Since $ \deg_{\cL_{1},\dots,\cL_{d-1}}(\div(s_{d})_{\hor,\eta})=
\deg_{\cL_{1},\dots,\cL_{d}}(\cW_{\eta})$, we obtain the result.
\end{proof}





\begin{thm}\label{thm:1} 
  Let $\cB$ be an arithmetic variety of relative dimension $b$ and
  $\ov \cH_{i}$, $i=1, \dots, b$, nef Hermitian line bundles on
  $\cB$. Let $ \KB=\field(\cB)$ be the function field of $\cB$ and
  $(\fM,\mu)$ the associated measured set of places as in
  \eqref{eq:6}.

  Let $\pi\colon \cX \to \cB$ be a dominant morphism of arithmetic
  varieties of relative dimension $n$ and $X$ the fibre of $\pi$ over
  the generic point of $\cB$.  Let $Y$ be a prime cycle of $X$ of
  dimension $d$ and $\cY$ its closure in $\cX$.  Let $\ov \cL_{j}$,
  $j=0,\dots,d$, be DSP Hermitian line bundles on $\cX$ and $\ov
  L_{j}$, $j=0,\dots,d$, the associated $\fM$-metrized line bundles as
  in \eqref{eq:20}.  Let $s_{0},\dots,s_{d}$ be rational sections of
  $\cL_{0},\dots ,\cL_{d}$ respectively, intersecting properly on $Y$.
  Then the function $\fM\to \R$ given, for $w\in\fM$, by
\begin{equation}\label{eq:37}
  w\longmapsto \h_{(L_{0},\|\cdot\|_{0,w}), \dots,
    (L_{d},\|\cdot\|_{d,w})}(Y;s_{0},\dots,s_{d})
\end{equation}
is $\mu$-integrable. Moreover,
  \begin{equation}\label{eq:7}
    \h_{\ov L_{0},\dots, \ov L_{d}}(Y)=
    \h_{\pi ^{\ast}\ov \cH_{_{1}},\dots,\pi ^{\ast}\ov \cH_{_{b}},\ov \cL_{0},\dots, \ov \cL_{d}}(\cY).
  \end{equation}
In other words, the integral of the function \eqref{eq:37} coincides with the height of $Y$
as defined in \eqref{eq:25}. 
\end{thm}

\begin{proof}
By linearity, we reduce to the case when the line bundles $\cL_{j}$
are ample, their metrics are semipositive and the sections are global
sections. Moreover, since multiplying one of the metrics on $\cL_{j}$
changes both sides of the equality \eqref{eq:7} by the same additive
constant, we can assume that the sections $s_{j}$ of $\cL_{j}$,
$j=0,\dots,d$ are small, in the sense that $\sup_{p\in
  X(\C)}\|s_{j}(p)\|_{j}\le 1$.       

  We proceed by induction on the dimension of $Y$. If $\dim(Y)=-1$,
  then $Y=\emptyset$ and so the local heights of $Y$ are zero. Hence,
  these local heights are $\mu$-integrable and, by Proposition \ref{prop:2},
  the equality in \eqref{eq:7} is reduced to $0=0$.

  We now assume that $\dim(Y)=d\ge 0$. In this case, the restriction
  $\pi|_{\cY}\colon \cY \to\cB$ is dominant. Since the height does not
  change by normalization, by restricting objects to $\cY$ and pulling back to
  its normalization, we may assume in the computations that follow
  that $\cY=\cX$. In particular, $Y=X$ and $d=n=\dim(X)$.

  Let $s_{0},\dots,s_{n}$ be global sections of $\cL_{0},\dots
  ,\cL_{n}$ respectively, that meet properly on $X$, and denote by
  $\rho \colon \fM\to \R$  the local height function in
  \eqref{eq:37}.  We have to show that this function is
  $\mu$-integrable and that
  \begin{displaymath}
    \int_{\fM}\rho (w)\dd \mu (w)=
    \h_{\pi ^{\ast}\ov \cH_{_{1}},\dots,\pi ^{\ast}\ov \cH_{_{b}},\ov
      \cL_{0},\dots, \ov \cL_{n}}(\cX).
  \end{displaymath}

  For each $w\in \fM$, by the definition of local heights in
  \eqref{eq:10}, we can write $ \rho(w) =\rho _{1}(w) -\rho _{2}(w) $
  with
  \begin{align*}
\rho_{1}(w)&=
    \h_{(L_{0},\|\cdot\|_{0,w}),\dots,
      (L_{n-1},\|\cdot\|_{n-1,w})}(\div(s_{n});s_{0},\dots,s_{n-1}),\\
\rho_{2}(w)&= \int_{X_{w}^{\an}}\log\|s_{n}\|_{n,w}\bigwedge_{j=0}^{n-1}
    \chern_{1}(L_{j},\|\cdot\|_{j,w}).
  \end{align*}
  We decompose the cycle $\div(s_{n})$ as
  \begin{displaymath}
    \div(s_{n})=\div(s_{n})_{\hor/\cB}+\div(s_{n})_{\vert/\cB}, 
  \end{displaymath}
  where $\div(s_{n})_{\hor/\cB}$ contains all the components that are
  dominant over $\cB$ and $\div(s_{n})_{\vert/\cB}$ contains the
  remaining ones. Clearly, $\div(s_{n})_{\hor/\cB}$ is the closure of
  $\div(s_{n})\cdot X$, and $\div(s_{n})_{\vert/\cB}$ contains all the
  components of $\div(s_{n})$ that do not meet $X$.

  By the inductive hypothesis, the function $w\mapsto \rho _{1}(w)$ is
  $\mu$-integrable and
  \begin{equation}\label{eq:5}
  \int_{\fM}  \rho _{1}(w)\dd \mu (w)=\h_{\pi ^{\ast}\ov
    \cH_{_{1}},\dots,\pi ^{\ast}\ov \cH_{_{b}},\ov \cL_{0},\dots, \ov
    \cL_{n}}(\div(s_{n})_{\hor/\cB}).
  \end{equation}
  Let now $w=\cV\in \cB^{(1)}$. The local ring $\cO_{\cB,\cV}$ is a
  discrete valuation ring. The scheme $\cX$ and the line bundle
  $\cL_{i}$ induce models $\cX_{\cV}$ and $\cL_{i,\cV}$ over
  $\Spec(\cO_{\cB,\cV})$ of $X$ and $L_{i}$.  Each component of the
  special fibre of $\cX_{\cV}$ is the localization
  $$\cW_{\cV}=\cW\underset{\cV}{\times}\Spec(K(\cV))$$ of a 
  hypersurface $\cW\in \cX^{(1)}$ with $\pi(\cW)=\cV$. Since the
  metric over $w$ is an algebraic metric coming from a model, by
  \cite[(1.3.6) and Remark 1.4.14]{BurgosPhilipponSombra:agtvmmh},
  \begin{equation}\label{eq:22}
    \rho _{2}(\cV)=
    - \sum_{\substack{\cW\in \cX^{(1)}\\\pi (\cW)=\cV}}\h_{\ov
      \cH_{1},\dots,\ov
      \cH_{b}}(\cV)\ord_{\cW}(s_{n})\deg_{\cL_{0},\dots,\cL_{n-1}}(\cW_{\cV}). 
  \end{equation}
  Since the number of components of $\div(s_{n})$ is finite, we
  deduce from \eqref{eq:22} that there is only a finite number of
  $\cV\in \cB^{(1)}$ with $\rho _{2}(\cV)\not = 0$. Thus $\rho _{2}$ is
  integrable on $\cB^{(1)}$ with respect to the counting measure $\mu_{\fin}$
  as in \eqref{eq:6}.  By Proposition \ref{prop:2},
  \begin{displaymath}
    \h_{\ov
      \cH_{1},\dots,\ov
      \cH_{b}}(\cV)\deg_{\cL_{0},\dots,\cL_{n-1}}(\cW_{\cV})=
     \h_{\pi ^{\ast}\ov \cH_{_{1}},\dots,\pi ^{\ast}\ov \cH_{_{b}},\ov
       \cL_{0},\dots, \ov \cL_{n-1}}(\cW). 
  \end{displaymath}
The same result implies that, if $\dim(\pi(\cW)) \le b-1$, then
  \begin{displaymath}
         \h_{\pi ^{\ast}\ov \cH_{_{1}},\dots,\pi ^{\ast}\ov \cH_{_{b}},\ov
       \cL_{0},\dots, \ov \cL_{n-1}}(\cW)=0.
  \end{displaymath}
  Since
  \begin{align*}
    \div(s_{n})_{\vert/\cB}&=
    \sum
    _{\substack{\cW\in \cX^{(1)}\\\dim (\pi (\cW))\le b}}\ord_{\cW}(s_{n})\cW\\
    &=\sum_{\cV\in \cB^{(1)}}\sum_{\substack{\cW\in \cX^{(1)}\\\pi
        (\cW)=\cV}}\ord_{\cW}(s_{n})\cW+
      \sum_{\substack{\cW\in \cX^{(1)}\\\dim (\pi (\cW))\le
        b-1}}\ord_{\cW}(s_{n})\cW,
  \end{align*}
  it follows from \eqref{eq:22} that
  \begin{align} \label{eq:15}
\nonumber     \int_{\cB^{(1)}} \rho _{2}(w) \dd \mu_{\fin} (w)=&\sum_{\cV\in
    \cB^{(1)}} \rho_{2}(\cV)\\ 
\nonumber =& -\sum_{\cV\in
    \cB^{(1)}}\sum_{\substack{\cW\in \cX^{(1)}\\\pi 
        (\cW)=\cV}}\ord_{\cW}(s_{n}) \h_{\pi ^{\ast}\ov
    \cH_{_{1}},\dots,\pi ^{\ast}\ov \cH_{_{b}},\ov \cL_{0},\dots, \ov
    \cL_{n}}(\cW) \\
=& -\h_{\pi ^{\ast}\ov
    \cH_{_{1}},\dots,\pi ^{\ast}\ov \cH_{_{b}},\ov \cL_{0},\dots, \ov
    \cL_{n}}(\div(s_{n})_{\vert/\cB}).
  \end{align}

  We next consider the places associated to the points $p\in
  \cB(\C)^{\gen}$. In this case, by the definition of $\rho_{2}$, we have that 
  \begin{equation*}
    \rho_{2} (p)= 
\int_{\pi^{-1}(p)}\log\|s_{n}\|_{n}\bigwedge_{j=0}^{n-1}
    \chern_{1}(\ov \cL_{j,\C}|_{\pi^{-1}(p)}) ,
  \end{equation*}
  where $\pi$ denotes the projection $\cX\to \cB$ and $\|\cdot\|_{n}$
  the metric in $\ov \cL_{n}$.  We have to show that $\rho_{2} $ is
  $\mu _{\infty}$-integrable with $\mu _{\infty}=\bigwedge_{i=1}^{b}
  \chern_{1}(\ov \cH_{i})$ and that
  \begin{equation} \label{eq:8}
\int_{\cB(\C)^{\gen}}\rho_{2} (p)\dd \mu _{\infty}(p)= 
  \int_{\cX(\C)}\log\|s_{n}\|_{n}\bigwedge_{j=0}^{n-1}\chern_{1}(\ov \cL_{j})
    \land \bigwedge_{i=1}^{b}
    \chern_{1}(\pi ^{\ast}\ov \cH_{i}).
  \end{equation}


We first assume that, for each $j=0,\dots,n $, the metric on the line
bundle $\cL_{j}$ is smooth, but  that the metric on
 $\cH_{i}$, $i=1,\dots, b$, is not necessarily smooth. By definition,  there is a
sequence of smooth semipositive metrics $(\|\cdot\|_{i,k})_{k\ge0}$
on $\cH_{i,\C}$ that converge to $\|\cdot \|_{i}$. Set $\ov
\cH_{i,k}=(\cH_{i},\|\cdot\|_{i,k})$ and  let $\mu _{\infty,k}$ be the
measure associated to the differential form
    \begin{displaymath}
      \chern_{1}(\ov \cH_{1,k})\land \dots \land
      \chern_{1}(\ov \cH_{b,k}). 
    \end{displaymath}
    By \cite[Th\'eor\`eme 4.1]{ChambertLoirThuillier:MMel}, the
    measures $\mu _{\infty,k}$ converge weakly to $\mu _{\infty}$.  By
    the same result, even if $\log\|s_{n}\|_{n}$ is not bounded, the
    equality
    \begin{multline*}
      \lim_{k\to \infty}
      \int_{\cX(\C)}\log\|s_{n}\|_{n}\bigwedge_{j=0}^{n-1}\chern_{1}(\ov
      \cL_{j}) 
    \land \bigwedge_{i=1}^{b}
    \chern_{1}(\pi ^{\ast}\ov \cH_{i,k})\\
    = 
    \int_{\cX(\C)}\log\|s_{n}\|_{n}\bigwedge_{j=0}^{n-1}\chern_{1}(\ov
    \cL_{j}) 
    \land \bigwedge_{i=1}^{b}
    \chern_{1}(\pi ^{\ast}\ov \cH_{i})
  \end{multline*}
  holds. Let $U\subset \cB(\C)$ be a connected Zariski open subset
  such that the
  restriction  $\pi\mid_{\pi^{-1}(U)}$ is a proper smooth map. By Ehresmann's
  fibration theorem, this restriction is a locally trivial
  proper differentiable fibration. Thus, that there exists a
  compact differentiable manifold 
  $F$ and an analytic open cover $\{U_{\alpha
  }\}_{\alpha}$ of $U$ such that $\pi ^{-1}(U_{\alpha })$ is diffeomorphic to
  $F\times U_{\alpha }$ for every $\alpha$. Let $\{\nu _{\alpha
  }\}_{\alpha }$ be a partition of
  unity subordinated to the open cover $\{U_{\alpha }\}_{\alpha }$. 

  Fix an $\alpha$.  To avoid burdening the notation, we identify $\pi
  ^{-1}(U_{\alpha })$ with $F\times U_{\alpha }$ through the above
  diffeomorphism.
Let $\lambda  _{F}$ be a measure of $F$ given by a
volume form. Since the metrics $\|\cdot\|_{j}$ are
  smooth, there is a smooth function $g\colon F\times U_{\alpha }\to \R$
  such that, for each $u\in U_{\alpha }$,
  \begin{displaymath}
    \bigwedge_{j=0}^{n-1}\chern_{1}(\ov \cL_{j})\, \bigg|_{\{u\}\times
      F}=g(\cdot,u)\lambda _{F}. 
  \end{displaymath}
  By \cite[Theorem 3.2]{Billinsley:cpm}, the measures $\lambda
  _{F}\times \bigwedge_{i=1}^{b}
  \chern_{1}(\pi ^{\ast}\ov \cH_{i,k})$ converge weakly to the
  measure $\lambda
  _{F}\times \bigwedge_{i=1}^{b}
  \chern_{1}(\pi ^{\ast}\ov \cH_{i})$. By the unicity of weak limits
  of measures,
  \begin{equation}\label{eq:9}
    \bigwedge_{j=0}^{n-1}\chern_{1}(\ov \cL_{j})
    \land \bigwedge_{i=1}^{b}
    \chern_{1}(\pi ^{\ast}\ov \cH_{i})\, \bigg|_{F\times U_{\alpha
      }}=g\lambda _{F}\times \bigwedge_{i=1}^{b}
    \chern_{1}(\pi ^{\ast}\ov \cH_{i}).
  \end{equation}
  Since $\log \|s_{n}\|_{n}$ is integrable with respect to $\bigwedge_{j=0}^{n-1}\chern_{1}(\ov \cL_{j})
    \land \bigwedge_{i=1}^{b}
    \chern_{1}(\pi ^{\ast}\ov \cH_{i})$, by \eqref{eq:9} the function $(\nu _{\alpha
    }\circ \pi )\log \|s_{n}\|_{n} \, g$ is integrable with respect to
    $\lambda _{F}\times \bigwedge_{i=1}^{b}
    \chern_{1}(\pi ^{\ast}\ov \cH_{i})$. By  Fubini's theorem
    \cite[Theorem 2.6.2]{Federer:GMT}, the function
    \begin{displaymath}
      \int_{F} (\nu _{\alpha
    }\circ \pi )\log \|s_{n}\|_{n} \, g \lambda _{F}=\nu _{\alpha }\rho _{2}
    \end{displaymath}
    is $\mu _{\infty }$-integrable and 
  \begin{align*}
\nonumber     \int_{U_{\alpha }}\nu _{\alpha }\rho_{2} (p)\dd \mu _{\infty}(p)&= 
  \int_{F\times U_{\alpha }}(\nu _{\alpha }\circ \pi
  )\log\|s_{n}\|_{n}\, g \lambda _{F}\times
    \bigwedge_{i=1}^{b}
    \chern_{1}(\pi ^{\ast}\ov \cH_{i}) \\ &=
  \int_{\pi ^{-1}(U_{\alpha })}(\nu _{\alpha }\circ \pi )\log\|s_{n}\|_{n}\bigwedge_{j=0}^{n-1}\chern_{1}(\ov \cL_{j})
    \land \bigwedge_{i=1}^{b}
    \chern_{1}(\pi ^{\ast}\ov \cH_{i}),
  \end{align*}
  where the last equality follows from \eqref{eq:9}.

  Since the above holds for every $\alpha$, Lebesgue's monotone
  convergence theorem \cite[Corollary 2.4.8]{Federer:GMT} and the fact
  that $\mu_{\infty(}\cB(\C)\setminus U))=0$ which follows from \cite[Corollaire
  4.2]{ChambertLoirThuillier:MMel}, imply that $\rho _{2}$ is $\mu
  _{\infty}$-integrable and that \eqref{eq:8} holds. Observe that we
  can apply Lebesgue's monotone convergence theorem because we are
  assuming that the section $s_{n}$ is small, and so the function
  $\log\|s_{n}\|_{n}$ is nonpositive.

  We now assume that the metrics on $\cL_{j}$ and $\cH_{i}$ are
not necessarily smooth, and choose sequences of smooth semipositive metrics
  $(\|\cdot\|_{j,k_{j}})_{k_{j}\ge0}$  on $\cL_{j}$ that converge
  uniformly to
  $\|\cdot\|_{j}$. For $p\in \cB(\C))^{\gen}$, write
    \begin{equation}\label{eq:12}
    \rho_{2,k_{0},\dots,k_{n}}
    (p)=\int_{\pi^{-1}(p)}\log\|s_{n}(p)\|_{n,k_{n}}\bigwedge_{j=0}^{n-1}\chern_{1}( \ov \cL_{j,k_{j},\C}|_{\pi^{-1}(p)}).
  \end{equation}
  From equation \eqref{eq:12} when $i=n$, and from equation
  \eqref{eq:12} and Stokes' theorem when $i\not = n$, one can prove
  that, for each
  $\varepsilon >0$, there is a constant $K_{i}$ that does not
  depend on $p$ nor on $k_{j}$, $j\not = i$, such that, for all
  $k_{i},k'_{i}\ge K_{i}$, 
  \begin{equation} \label{eq:11}
    |\rho _{2,k_{0},\dots,k_{i},\dots,k_{n}}(p)-\rho
    _{2,k_{0},\dots,k'_{i},\dots,k_{n}}(p)|\le \varepsilon.
  \end{equation}

  For $k\ge 0$, denote by  $\rho_{2,k}$  the function in \eqref{eq:12} for the
  choice of indices $k_{0}=\dots=k_{n}=k$. We deduce from \eqref{eq:11} that the
  diagonal sequence $(\rho _{2,k})_{k\ge 0}$ converges uniformly to
  $\rho _{2}$. Since the measure $\mu_{\infty}$ has finite total mass
  and, by the previous case, the functions $\rho _{2,k}$ are $\mu
  _{\infty }$-integrable, we deduce that $\rho_{2} $ is $\mu _{\infty
  }$-integrable and that
  \begin{equation*}
    \lim_{k\to \infty} \int_{\cB(\C)}\rho_{2,k}(p)\dd \mu_{\infty}
    (p)=
    \int_{\cB(\C)}\rho _{2}(p)\dd \mu_{\infty} (p).
  \end{equation*}
  Therefore, using \eqref{eq:8} for the functions $\rho _{2,k}$ and
  \cite[Th\'eor\`eme 4.1]{ChambertLoirThuillier:MMel}, we deduce
  that \eqref{eq:8} also holds in the case when all the metrics are
  semipositive.

  In consequence,  $\rho =\rho _{1}-\rho _{2}$ is $\mu
  $-integrable and, using \eqref{eq:5}, \eqref{eq:15}, \eqref{eq:8},
  the arithmetic B\'ezout theorem in \eqref{eq:10} and the inductive hypothesis, 
  \begin{align*}
    \int_{\fM}\rho (w)\dd \mu (w)&= \h_{\pi ^{\ast}\ov
      \cH_{_{1}},\dots,\pi ^{\ast}\ov \cH_{_{b}},\ov \cL_{0},\dots, \ov
      \cL_{n}}(\div(s_{n})_{\hor/\cB})\\
    &\phantom{=\ }
    + \h_{\pi ^{\ast}\ov
      \cH_{_{1}},\dots,\pi ^{\ast}\ov \cH_{_{b}},\ov \cL_{0},\dots, \ov
      \cL_{n}}(\div(s_{n})_{\vert/\cB})\\
    &\phantom{=\ }-
    \int_{\cX(\C)}\log\|s_{n}\|_{n}\bigwedge_{j=0}^{n-1}\chern_{1}(\ov \cL_{j})
    \land \bigwedge_{i=1}^{b}
    \chern_{1}(\pi ^{\ast}\ov \cH_{i})\\
    &=\h_{\pi ^{\ast}\ov \cH_{_{1}},\dots,\pi ^{\ast}\ov \cH_{_{b}},\ov
      \cL_{0},\dots, \ov \cL_{n}}(\cX),
  \end{align*}
which concludes the proof. 
\end{proof}

\begin{exmpl} \label{exm:3} Let $\K\simeq \Q(z_{1},\dots, z_{b})$
 with the measured set of places $(\fM,\mu)$ as in Examples
  \ref{exm:1} and \ref{exm:2}.
Let
\begin{displaymath}
  \cX=\P^{b}_{\Z} \times
  \P_{\Z}^{1} \quad \text{ and } \quad \ov \cL=
  \varpi^{*}\ov{\cO_{\P^{1}_{\Z}}(1)}^{\can},   
\end{displaymath}
where $\varpi$ denotes the projection $ \P^{b}_{\Z}\times \P_{\Z}^{1}
\to \P_{\Z}^{1}$, and let $\ov L={\ov {\cO_{\P^{1}_{\K}}(1)}^{\can}}$
denote the canonical $\fM$-metrized line bundle structure on the
universal line bundle of $\P^{1}_{\K}$ as in Example \ref{exm:4}. 

Let $(1:\gamma) \in \P^{1}_{\Z}(\K)$ with $\gamma\in \K^{\times}$.
The closure $\cY $ of this point in $\cX$ is the hypersurface defined
by the bihomogenization of the polynomial $P_{\gamma}$ in
\eqref{eq:26}.  In this case, Theorem \ref{thm:1} together with
\eqref{eq:23} and \eqref{eq:33} gives
\begin{displaymath}
\h_{\pi ^{\ast}\ov
    \cH_{_{1}},\dots,\pi ^{\ast}\ov {\cH_{b}},\ov \cL}(\cY) =\h_{\ov L}(1:\gamma)=
  \t_{\K,\fM,\mu}(\gamma)=   \m(P_{\gamma}),
\end{displaymath}
where $\m(P_{\gamma})$ denotes the logarithmic Mahler measure
of $P_{\gamma}$.
\end{exmpl}

\section{Height of toric varieties over finitely generated
  fields} \label{sec:toric-varieties} 

Using our previous work on toric varieties in
\cite{BurgosPhilipponSombra:agtvmmh}, we can give a ``combinatorial''
formula for the mixed height of a toric variety with respect to a
family of $\fM$-metrized line bundles. As a consequence of Theorem
\ref{thm:1}, this formula also expresses an arithmetic intersection number, in the sense
of Gillet-Soul\'e, of a non toric arithmetic variety.

Let $\cB$ be an arithmetic variety of relative dimension $b$ and $\ov
\cH_{i}=(\cH_{i},\|\cdot\|_{i})$, $i=1,\dots,b$, a family of nef
Hermitian line bundles on $\cB$, as at the beginning of
\S~\ref{sec:general-varieties}. Let $\K=\KK(\cB)$ and $(\fM,\mu)$ the
associated set of places of $\K$ as in \eqref{eq:6}.

Let $\T\simeq \G_{m}^{n}$ be a split torus of dimension $n$ over
$\K$. Let $N=\Hom(\G_{m},\T)$ be the lattice of cocharacters of $\T$,
$M=\Hom(\T,\G_{m})=N^{\vee}$ the lattice of characters, and set
$N_{\R}=N\otimes \R$ and $M_{\R}=M\otimes \R$.

Let $X$ be a proper toric variety over $\K$ with torus $\T$, described
by a complete fan $\Sigma$ on $N_{\R}$. 
A toric divisor on $X$ is a Cartier divisor invariant under the action
of $\T$. Such a divisor $D$ defines a ``virtual support function'',
that is, a function $\Psi_{D}\colon N_{\R}\to \R$ whose restriction to
each cone of the fan~$\Sigma$ is an element of $M$.  The toric divisor
$D$ is nef if and only if $\Psi_{D}$ is concave.  One can also
associate to $D$ the polytope defined as 
\begin{displaymath}
  \Delta_{D}=\{x\in M_{\R}\mid x\ge \Psi_{D}\}.
\end{displaymath}

Now let
\begin{displaymath}
 \pi \colon \cX\to \cB 
\end{displaymath}
be a dominant morphism of arithmetic varieties of relative dimension
$n\ge 0$ and $\ov \cL$ a Hermitian line bundle on $\cX$. We assume
that $(X,L)$, the fibre of $(\cX,\cL)$ over the generic point of
$\cB$, is a toric variety over $\K$ with a line bundle associated to a
toric divisor $D$ on $X$. We consider the associated $\fM$-metrized
line bundle $\ov L$ on $X$ as in~\eqref{eq:20}.

For each place $w\in \mathfrak{M}$, we associate to the 
torus $\T$ an analytic space $\T^{\an}_{w}$ and we denote by
$\SS_{w}$ its compact subtorus. In the Archimedean case, it is
isomorphic to~$(S^{1})^{n}$. In the non-Archimedean case, it is a
compact analytic group, see \cite
[\S~4.2]{BurgosPhilipponSombra:agtvmmh} for a description.  Then, the 
$\fM$-metrized line bundle $\ov L= \ov{\cO(D)}$ on $X$ is \emph{toric}
if  its 
 $w$-adic metric $\|\cdot\|_{w}$  is invariant with respect
to the action of~$\SS_{w}$ for all $w$.

Assume that $\ov L$ is toric and let $s$ be the toric section of $L$
with $\div(s)=D$. For each $w\in
\fM$, denote by
$X^{\an}_{0,w}$ the analytification of the open principal subset
$X_{0}\subset X$  corresponding to the cone $\{0\}$, which is
isomorphic to the torus $\T$. Then the function $X^{\an}_{0,w}\to \R$
given by $p\mapsto 
\log\|s(p)\|_{w}$ is invariant under the action of~$\SS_{w}$ and
induces a function $\psi _{\ov L,s,w}\colon N_{\R}\to \R$ as in 
\cite[Definition 4.3.5]{BurgosPhilipponSombra:agtvmmh}. For shorthand,
when $\ov L$ and $s$ are fixed, we will denote $\psi _{\ov L,s,w}$ by
$\psi _{w}$. 

We now further assume that the line bundle $\cL$ is generated by
global sections and that the Hermitian metric on $\ov \cL$ is
semipositive. Hence, the line bundle $L$ is also generated by global
sections and, for each $w\in \fM$, the metric induced in $L_{w}^{\an}$
by $\ov \cL$ is semipositive. In this case, by \cite[Theorem
4.8.1]{BurgosPhilipponSombra:agtvmmh}, for each $w\in \fM$, the
function $\psi _{w}$ is concave. We associate to it a concave function
on $\Delta _{D}$, denoted by $ \vartheta _{\ov L,s,w}$ or $\vartheta
_{w}$ for short, and called the $w$-adic \emph{roof function} of the
pair $(\ov L,s)$ as in \cite[Definition
5.1.4]{BurgosPhilipponSombra:agtvmmh}.  This function is defined as
the Legendre-Fenchel dual of~$\psi_{w}$, and so it is defined, for $x\in
\Delta_{D}$, as
\begin{displaymath}
  \vartheta _{w}(x)=
  \inf_{u\in N_{\R}}(\langle x,u\rangle -\psi _{w}(u)).
\end{displaymath}

We denote by $\vol_{M}$ the measure on $\Delta_{D}$ given by the
restriction of the Haar measure on $M_{\R}$ normalized so that the
lattice $M$ has covolume $1$.

\begin{cor}
  \label{cor:1}
With notation as above, the function 
\begin{equation}\label{eq:39}
  \fM\longrightarrow \R, \quad w\longmapsto
  \int_{\Delta_{D}}\vartheta_{\ov L,s,w}(x)\dd\vol_{M}(x)
\end{equation}
is $\mu$-integrable. Moreover, 
\begin{multline}\label{eq:40}
    \h_{\pi ^{\ast}\ov \cH_{1},\dots,\pi ^{\ast}\ov \cH_{b},\ov \cL,\dots,\ov \cL}(\cX)=\h_{\ov L}(X)=
(n+1)!\int_{\fM}  \int_{\Delta_{D}}\vartheta_{\ov L,s,w}(x)\dd\vol_{M}(x)\dd\mu(w).
\end{multline}
\end{cor}

\begin{proof}
By \cite[Theorem 5.1.6]{BurgosPhilipponSombra:agtvmmh}, 
the quantity
\begin{displaymath}
 (n+1)!\int_{\Delta_{D}}\vartheta_{\ov
  L,s,w}(x)\dd\vol_{M}(x) 
\end{displaymath}
  is equal to the difference of local heights
\begin{equation*}
\h_{(L_{0},\|\cdot\|_{0,w}), \dots,
    (L_{n},\|\cdot\|_{n,w})}(Y;s_{0},\dots,s_{n})
-\h_{(L_{0},\|\cdot\|_{0,w,\can}), \dots,
    (L_{n},\|\cdot\|_{n,w,\can})}(Y;s_{0},\dots,s_{n}),
\end{equation*}
where $\|\cdot\|_{i,w,\can}$ is the canonical $w$-adic metric on
$L_{i}$ as in \cite[Proposition-Definition
4.3.15]{BurgosPhilipponSombra:agtvmmh}. By Theorem \ref{thm:1}, both
local heights are $\mu$-integrable. Hence, so is the function in
\eqref{eq:39}, which proves the first statement.

For the second statement, the first equality  follows
from Theorem \ref{thm:1}. By the discussion above, 
\begin{displaymath}
  (n+1)!\int_{\fM}  \int_{\Delta_{D}}\vartheta_{\ov
    L,s,w}(x)\dd\vol_{M}(x)\dd\mu(w) = \h_{\ov L}(X)-\h_{\ov L^{\can}}(X).
\end{displaymath}
Using the argument
   in the proof of \cite[Proposition
   5.2.4]{BurgosPhilipponSombra:agtvmmh}, it can be shown that
   $\h_{\ov L^{\can}}(X)=0$, which proves the second equality in \eqref{eq:40}.
\end{proof} 

\begin{thm}\label{thm:2} Let notation be as above.
  \begin{enumerate}
  \item \label{item:1} For each $x\in \Delta _{D}$, the function
       \begin{displaymath}
      \fM\longrightarrow  \R, \quad w\longmapsto \vartheta _{w}(x)
    \end{displaymath}
is  $\mu$-integrable. 
  \item \label{item:2} The function
    \begin{displaymath}
\Delta_{D}\longrightarrow \R,\quad      x\longmapsto \int_{\fM} \vartheta _{w}(x)\dd \mu (w)
    \end{displaymath}
    is concave and continuous on $\Delta _{D}$.
  \item \label{item:3} The function
    \begin{displaymath}
  \fM\times \Delta_{D}\longrightarrow
\R, \quad (w,x)\longmapsto\vartheta _{w}(x)     
    \end{displaymath}
  is $(\mu\times\vol_{M})$-integrable. 
  \end{enumerate}
\end{thm}

\begin{proof}
Let $\sigma \in \Sigma^{n}$. The closure  $V(\sigma)$
of the orbit of $X$ corresponding to $\sigma$ is a
point. By \cite[Proposition 4.8.9]{BurgosPhilipponSombra:agtvmmh}, for
each $w\in \fM$,
\begin{displaymath}
  \vartheta_{\iota^{*} \ov L, \iota^{*}s_{\sigma},w}(0)=
  \vartheta_{\ov L, s_{\sigma},w}(m_{\sigma}) =\vartheta_{w}(m_{\sigma}) , 
\end{displaymath}
where $\iota$ denotes the inclusion $ V(\sigma)\hookrightarrow X$.
By Corollary \ref{cor:1}, the function $w\mapsto
\vartheta_{w}(m_{\sigma})$ in $\mu$-integrable, and its integral
coincides with the height of $V(\sigma)$ with respect to $\ov L$.

Since $\vartheta_{w}$ is a concave function, for all $x\in\Delta_{D}$,
\begin{equation}
  \label{eq:35}
  \min_{\sigma\in \Sigma^{n}} \vartheta_{w}(m_{\sigma}) = \min_{y\in \Delta_{D}} \vartheta_{w}(y) \le \vartheta_{w}(x).  
\end{equation}
On the other hand, using again the concavity of $\vartheta_{w}$, 
\begin{align}
  \label{eq:36}
\nonumber   \vartheta_{w}(x)- \min_{y\in \Delta_{D}} \vartheta_{w}(y)& \le
  \max_{y\in \Delta_{D}} \vartheta_{w}(y)-\min_{y\in \Delta_{D}}
  \vartheta_{w}(y)\\
&\le \frac{n+1}{\vol_{M}(\Delta_{D})}
  \int_{\Delta_{D}}\Big( \vartheta_{w}(z) - \min_{y\in \Delta_{D}}
  \vartheta_{w}(y) \Big) \dd\vol_{M}(z) 
\end{align}
It follows from \eqref{eq:35} and \eqref{eq:36} that, for all $x\in \Delta_{D}$,
\begin{displaymath}
  \min_{\sigma\in \Sigma^{n}} \vartheta_{w}(m_{\sigma})  \le
  \vartheta_{w}(x) 
  \le \frac{n+1}{\vol_{M}(\Delta_{D})}
  \int_{\Delta_{D}}\vartheta_{w}(z) \dd\vol_{M}(z) - n \min_{\sigma\in \Sigma^{n}} \vartheta_{w}(m_{\sigma}).
\end{displaymath}
By Corollary \ref{cor:1} and the fact that $\Delta_{D}$ has finite
measure, we have that both the upper and the lower bound are
integrable with respect to the measure $\mu\times\vol_{M}$. The
statements \eqref{item:1} and \eqref{item:3} follow directly from
these bounds, while the statement \eqref{item:2} follows from the same
bounds and Lebesgue's bounded convergence theorem \cite[Theorem
2.4.9]{Federer:GMT}.
\end{proof}

\begin{defn}\label{def:2}
With notations as above,   the \emph{(global) roof function} is the continuous concave function
  $\vartheta_{\ov L,s} \colon \Delta _{D}\to \R$ given by
  \begin{displaymath}
    \vartheta_{\ov L,s} (x)=\int_{\fM}\vartheta _{w}(x)\dd \mu (w).
  \end{displaymath}
\end{defn}

\begin{cor}\label{cor:2}
  With the previous notations, 
  \begin{displaymath}
    \h_{\pi ^{\ast}\ov \cH_{1},\dots,\pi ^{\ast}\ov \cH_{b},\ov \cL,\dots,\ov \cL}(\cX)=\h_{\ov L}(X)
=(n+1)!\int_{\Delta_{D}}\vartheta_{\ov L,s}(x) \dd
    \vol_{M}(x)
  \end{displaymath}
  holds.
\end{cor}
\begin{proof}
This follows from Corollary \ref{cor:1} and Theorem
\ref{thm:2}\eqref{item:3} together with Fubini's theorem \cite[Theorem
  2.6.2]{Federer:GMT}. 
\end{proof} 

\begin{rem}\label{rem:2} More generally, when we have a family $\ov
  \cL_{0},\dots,\ov \cL_{n}$ of semipositive Hermitian line bundles on
  $\cX$ such that the induced Hermitian line bundles $\ov
  L_{0},\dots,\ov L_{n}$ on $X$ are toric, we can express
  \begin{displaymath}
    \h_{\pi ^{\ast}\ov \cH_{1},\dots,\pi ^{\ast}\ov \cH_{b},\ov
      \cL_{0},\dots,\ov \cL_{n}}(\cX)=\h_{\ov L_{0},\dots, \ov L_{n}}(X) 
  \end{displaymath}
 in terms of mixed integrals, similarly as in
  \cite[Theorem~5.2.5]{BurgosPhilipponSombra:agtvmmh},
\end{rem}
 
\section{Canonical height  of translated of subtori over finitely
  generated  fields}
\label{sec:height-transl-subt}

In this section, we particularize the formulae in
\S~\ref{sec:toric-varieties} to the case when $X$ is the normalization
of a translated of a subtori in the projective space.

As before, let $\cB$ be an arithmetic variety of relative dimension $b$ and $\ov
\cH_{i}=(\cH_{i},\|\cdot\|_{i})$, $i=1,\dots,b$, a family of nef
Hermitian line bundles on $\cB$. Let also $\K=\KK(\cB)$ and $(\fM,\mu)$ the
associated set of places of $\K$ as in \eqref{eq:6}.

Let $r\ge1$ and consider the projective space $\P^{r}_{\cB}$ over
$\cB$ and the universal line bundle $\cO_{\P^{r}_{\cB}}(1)$ on
it. Since $\P^{r}_{\cB}=\P^{r}_{\Z}\underset{\Spec(\Z)}{\times}\cB$
and $\cO_{\P^{r}_{\cB}}(1)$ is the pull-back of $\cO_{\P^{r}_{\Z}}(1)$
under the first projection, we can pull-back the canonical metric on
$\cO_{\P^{r}_{\Z}}(1)$ to obtain a metric on $\cO_{\P^{r}_{\cB}}(1)$,
also called canonical. We denote by $\ov{\cO(1)}=\ov{\cO_{\P^{r}_{\cB}}(1)}$ the
obtained Hermitian line bundle.  

Choose $\bfm_{j}\in \Z^{n}$ and $f_{j}\in \K^{\times}$, $j=0,\dots,
r$. For simplicity, we assume that $\bfm_{0}=0\in \Z^{n}$ and that the
collection of vectors $\bfm_{j}$ generates $\Z^{n}$ as Abelian group.
Consider the map
\begin{displaymath}
 \G^{n}_{m,\KB}\longrightarrow
\P^{r}_{\KB}, \quad 
  \bft\longmapsto (f_{0}\bft^{\bfm_{0}}:\dots:f_{r}\bft^{\bfm_{r}}),
\end{displaymath}
where $f_{j}\bft^{\bfm_{j}}$ denotes the monomial
\begin{math}
   f_{j}t_{1}^{m_{j,1}}\dots t_{n}^{m_{j,n}}.
\end{math}
 We then denote by $Y$  the closure in  $\P^{r}_{\KB}$  of the image
 of this  map. 

The projective space $\P^{r}_{\KB}$ is the fibre of $\P^{r}_{\cB}$ over the
generic point of $\cB$. We denote by $\cY$ the closure of $Y$ in
$\P^{r}_{\cB}$ and by $\pi\colon \cY \to \cB$ the dominant map obtained by
restricting the projection $\P^{r}_{\cB}\to \cB$.
In this setting, we want to give a formula for the arithmetic intersection number
\begin{equation}\label{eq:2}
      \h_{\pi ^{\ast}\ov \cH,\dots,\pi ^{\ast}\ov \cH,\ov{\cO(1)}, \dots, \ov{\cO(1)}}(\cY).
\end{equation}

The subvariety $Y$ is not a toric variety over $\K$ because it is not
necessarily normal. Indeed, it is a ``translated toric subvariety'' of
$\P^{r}_{\KB}$ in the sense of
\cite[Definition~3.2.6]{BurgosPhilipponSombra:agtvmmh}.  Let $\cX$ be
the normalization of $\cY$, and $X$ the corresponding variety over
$\KB$. Let $\ov \cL$ be the pull-back of $\ov {\cO(1)}$ to $\cX$ and
$\ov L$ the associated $\fM$-metrized line bundle over $X$ as in
\eqref{eq:20}.  Therefore, $X$ is a toric variety over $\KB$ with
torus $ \G^{n}_{m,\KB}$ and the $\fM$-metrized line bundle $\ov L$ is
toric and semipositive.
 
In this case we can give an explicit description of the corresponding
$w$-adic roof functions.

\begin{prop} \label{prop:3} With notation as above, let $s$ be the
  toric section of $L$ determined by the section $x_{0}$ of
  $\cO(1)$. The polytope associated to the divisor $D=\div(s)$ on $X$
  is given by
  \begin{displaymath}
    \Delta =\conv(\bfm_{0},\dots,\bfm_{r})
  \end{displaymath}
  and, for $w\in \fM$, the $w$-adic roof function $\vartheta_{w}\colon
  \Delta\to \R$ is the function parameterizing the upper envelope of
  the extended polytope $\wt\Delta_{w}\subset\R^{n}\times \R$ given by
  \begin{displaymath}
  \wt\Delta_{w}=\begin{cases}
   \conv\big((\bfm_{j},-\h_{\ov \cH_{1},\dots,\ov \cH_{b}}(\cV)\ord_{\cV}(f_{j}))_{j=0,\dots, r}\big),
&\text    { 
      if }w=\cV\in \cB^{(1)}, \\
    \conv((\bfm_{j},\log|f_{j}(p)|)_{j=0,\dots, r}),&\text{
      if }w=p\in \cB(\C)^{\gen}.
  \end{cases}   
  \end{displaymath}
\end{prop}

\begin{proof}
  This follows from \cite[Example
  5.1.16]{BurgosPhilipponSombra:agtvmmh}. 
\end{proof}

Combining this result with Corollary \ref{cor:1}, we obtain a formula
for the arithmetic intersection number in \eqref{eq:2}, that we will
particularize to concrete examples.

\begin{cor}
  \label{cor:3}
  With notation as above, set $\cB=\P^{b}_{\Z}$ and $\ov \cH_{i}=\ov
  \cH=\ov{\cO_{\P_{\cB}^{r}}(1)}^{\can}$, $i=1,\dots, b$. Then
  $\h_{\pi ^{\ast}\ov \cH,\dots,\pi ^{\ast}\ov \cH,\ov{\cO(1)}, \dots,
    \ov{\cO(1)}}(\cY)$ is equal to
\begin{equation*}
(n+1)!\Bigg(\int\Bigg(\int_{\Delta}\vartheta
  _{p}(x)\dd \vol(x)\Bigg) \dd \mu _{\SS}
  +\sum_{\cV\in \cC}\int_{\Delta}\vartheta
  _{\cV}(x)\dd \vol(x) \Bigg),
\end{equation*}
where $\cC\subset \cB^{(1)}$ is the set of irreducible components of
the divisors $\div(f_{j})$, $j=0,\dots,r$, $\vol$ denotes the
Lebesgue measure on $\R^{n}$, and $\mu _{\SS}$ is the Haar measure of
the compact torus $\SS$ as in Example \ref{exm:1}.
\end{cor}

\begin{proof}
  We have that $\h_{\pi ^{\ast}\ov \cH,\dots,\pi ^{\ast}\ov
    \cH,\ov{\cO(1)}, \dots, \ov{\cO(1)}}(\cY)= \h_{\pi ^{\ast}\ov
    \cH,\dots,\pi ^{\ast}\ov \cH,\ov{\cL}, \dots, \ov{\cL}}(\cX)$ by
  the invariance of the height under normalization.  The formula then
  follows from Corollary \ref{cor:1},  Proposition \ref{prop:3}, the
  description of the measured set of places $(\fM,\mu)$ of the field
  $\KB$ in Example \ref{exm:1} together with the fact that, for $\cV
  \in \cB^{(1)}\setminus \cC$, the local roof function $\vartheta
  _{\cV}$ vanishes identically.
\end{proof}

\begin{exmpl}\label{exm:5}
  Consider the case $n=0$. Thus $\bfm_{\ell}=0$ for all $\ell$ and
  choose a collection $f_{0},\dots,f_{r}\in \Z[z_{1},\dots,z_{b}]$ of
  coprime polynomials with integer coefficients.  Hence
  $\Delta=\{0\}\subset \R^{0}$ and, for $w\in \fM$, the local roof
  function is given by 
  \begin{align*}
          \vartheta _{w}(0)
=    \begin{cases}
\max_{\ell}\log|f_{\ell}(p)| & \text{ if  }w=p\in
    \P_{\Z}^{b}(\C)^{\gen},\\ 
    -\h(\cV)\min_{\ell}\ord_{\cV}(f_{\ell}) & \text{ if }w=\cV\in
    (\P_{\Z}^{b})^{(1)}. 
    \end{cases}
  \end{align*}
  If $\cV$ is not the hyperplane at infinity of  $\P^{b}$, then by the
  coprimality of the $f_{\ell}$, we have
  $\min_{\ell}\ord_{\cV}(f_{\ell})=0$, while, if $\cV$ is the hyperplane
  at infinity, $\h(\cV)=0$. Thus the finite contribution vanishes and,
  from   Theorem \ref{thm:1} and Proposition \ref{prop:3} we deduce
  \begin{displaymath}
    \h_{\pi ^{\ast}\ov \cH,\dots,\pi ^{\ast}\ov \cH,\ov
      \cL}(\cY)=
    \int\max_{\ell}(\log|f_{\ell}(p)|)\dd \mu _{\SS}.
  \end{displaymath}

  In particular, if $r=1$, this arithmetic intersection number agrees
  with the size of the element $\gamma =f_{1}/f_{0}$ given in Example
  \ref{exm:2}. For instance, consider the case when  $b=1$, $f_{0}=\alpha$  and
  $f_{1}=\beta z_{1}+\gamma $ with $\alpha, \beta, \gamma$ coprime
  integers. Using \eqref{eq:33} and \eqref{eq:26}, the
  corresponding intersection number is given by the logarithmic Mahler measure
  of the affine polynomial $\alpha t_{1}-\beta z_{1}-\gamma$. By
  \cite[Proposition 7.3.1]{Maillot:GAdvt}, this logarithmic Mahler
  measure can be computed in terms of the Bloch-Wigner dilogarithm. 
\end{exmpl}

\begin{exmpl}
  \label{exm:8}
  Let $n=1$ and consider the case when $m_{i}=i$, $i=0,\dots, r$, and
  $f_{0},\dots,f_{r}\in \Z[z_{1},\dots,z_{b}]$ is a family of coprime
  polynomial with integer coefficients with $f_{0}=f_{r}=1$. 

We have that $\Delta=[0,r]$.
Let $w\in \fM$ and
  $\vartheta_{w}\colon[0,r]\to\R$ the corresponding local roof
  function. By Proposition \ref{prop:3}, if $w=\cV\in
  (\P^{b}_{\Z})^{(1)}$, then this function is zero. Also, if $w=p\in
  \P^{b}_{\Z}(\C)^{\gen}$,  this function is the 
minimal concave function on $[0,r]$ whose values at the
integers are given, for $i=0,\dots, r$, by
\begin{equation*}
\vartheta_{w}(i)= \mathop{\max_{0\le j\le i\le \ell\le r}}_{j\ne \ell} \bigg(
  \frac{\ell-i}{\ell-j} \log|f_{j}(p)|_{w}+ \frac{i-j}{\ell-j} \log|f_{\ell}(p)|_{w} \bigg).
\end{equation*}
In particular, $\vartheta_{w}(0)=\log|f_{0}(p)|_{w} =0$ and
$\vartheta_{w}(r)=\log|f_{r}(p)|_{w} =0$. 
It follows that
\begin{equation}\label{eq:42}
 \int_{0}^{r}\vartheta_{w}(x)\dd x=  \sum_{i=1}^{\delta-1}  \max_{j\le i\le \ell, j\ne \ell} \bigg(
  \frac{\ell-i}{\ell-j} \log|f_{j}(p)| + \frac{i-j}{\ell-j} \log|f_{\ell}(p)| \bigg) 
\end{equation}
From   Corollary \ref{cor:3} and
 \eqref{eq:42}, we deduce that
  \begin{equation*}
   \h_{\pi ^{\ast}\ov \cH,\dots,\pi ^{\ast}\ov \cH,\ov
      \cL}(\cY) =
2    \int \sum_{i=1}^{r-1}  \mathop{\max_{0\le j\le i\le \ell\le r}}_{j\ne \ell} \bigg(
  \frac{\ell-i}{\ell-j} \log|f_{j}(p)|  + \frac{i-j}{\ell-j}
  \log|f_{j}(p)| \bigg) \dd \mu_{\SS}.
  \end{equation*}
  Hence, this arithmetic intersection number can  be expressed in
  terms of integrals over the compact torus, of maxima of logarithms of
  absolute values of polynomials.
\end{exmpl}





\newcommand{\noopsort}[1]{} \newcommand{\printfirst}[2]{#1}
  \newcommand{\singleletter}[1]{#1} \newcommand{\switchargs}[2]{#2#1}
  \def\cprime{$'$}
\providecommand{\bysame}{\leavevmode\hbox to3em{\hrulefill}\thinspace}
\providecommand{\MR}{\relax\ifhmode\unskip\space\fi MR }
\providecommand{\MRhref}[2]{%
  \href{http://www.ams.org/mathscinet-getitem?mr=#1}{#2}
}
\providecommand{\href}[2]{#2}

\end{document}